\documentclass[12pt,oneside]{amsart}

\usepackage{amsmath}

\calclayout

\usepackage{amsfonts,amssymb,amsthm,enumitem}
\usepackage{mathtools}
\usepackage{nicefrac,setspace}
\usepackage{xcolor}
\usepackage{tikz}
\usepackage{circuitikz}

\usepackage{bm}

\usepackage{hyperref}       
\usepackage{url}            
\usepackage{booktabs}       
\usepackage{amsfonts}       
\usepackage{nicefrac}       
\usepackage{microtype}      

\usepackage{amsthm}

\theoremstyle{plain}
\newtheorem{theorem}{Theorem}[section]
\newtheorem{lemma}[theorem]{Lemma}
\newtheorem{corollary}[theorem]{Corollary}

\newtheorem{question}[theorem]{Question}

\newtheorem*{theorem*}{Theorem}
\newtheorem*{lemma*}{Lemma}
\newtheorem*{corollary*}{Corollary}
\newtheorem*{claim*}{Claim}

\theoremstyle{definition}
\newtheorem{definition}[theorem]{Definition} 

\theoremstyle{remark}
\newtheorem*{remark}{Remark}                

\newcommand{\N}{{\mathbb N}}

\newcommand{\R}{\mathbb{R}}

\newcommand{\gr}{\omega} 

\newcommand{\showcomments}{yes}

\newsavebox{\commentbox}
{\ifthenelse{\equal{\showcomments}{yes}}%
{\footnotemark
    \begin{lrbox}{\commentbox}
    \begin{minipage}[t]{1.25in}\raggedright\sffamily\tiny
    \footnotemark[\arabic{footnote}]}
{\begin{lrbox}{\commentbox}}}%
{\ifthenelse{\equal{\showcomments}{yes}}%
{\end{minipage}\end{lrbox}\marginpar{\usebox{\commentbox}}}
{\end{lrbox}}}

\subjclass[2020]{
60B20, 
 05C80, 
20E07, 
20E05, 
20F65. 
}
\keywords{Growth of groups, free groups, random graphs, non-backtracking matrix}
\date{\today}

\thanks{M.L.  received funding  from the European Union’s Horizon 2020 research and innovation program under the
Marie Sklodowska-Curie grant agreement No 101034255.
D.T.W.  supported by NSERC}

\title{Subgroups of a free group with every growth rate}
\author{Michail Louvaris}
\address{Laboratoire d’informatique Gaspard Monge (LIGM / UMR 8049), Université Gustave Eiffell, Marne-la-Vallée}
\email{michail.louvaris@univ-eiffel.fr}

\author{Daniel T. Wise}
\address{Department of Mathematics, Weizmann Institute of Science, Rehovot, Israel}
\email{daniel.wise@weizmann.ac.il}

\author{Gal Yehuda}
\address{Dept of Math, Yale, New Haven, USA}
\email{gal.yehuda@yale.edu}

\begin{document}
\vspace*{-1cm} 


\begin{abstract}
For every $\alpha \in [1,2r-1]$, we show there exists a subgroup $H<F_r$ whose growth rate is $\alpha$. 
\end{abstract}

\maketitle

\vspace{-.4cm}

\section{Introduction}
An interesting aspect of the coarse geometry of a finitely generated group~$\Gamma$
is its \emph{growth function}
\[
f_{(\Gamma,S)}(n)=\#\bigl\{g\in\Gamma : \|g\|_{S}\le n\bigr\},
\]
defined with respect to a finite generating set~$S$ and the associated word length~$\|g\|_{S}$.

The subject originated in Milnor’s comparison of Riemannian volume growth with word-metric growth of a fundamental group.
Wolf refined this viewpoint, Gromov proved the landmark theorem that a finitely generated group has polynomial growth precisely when it is virtually nilpotent, and Grigorchuk’s celebrated 1984 example revealed  the landscape also contains groups of intermediate (subexponential) growth.

Our focus will be the \emph{exponential growth rate}
\[
\operatorname{gr}(\Gamma,S) 
\ = \ 
\lim_{n\to\infty} f_{(\Gamma,S)}(n)^{1/n} \ \in \ [1,\infty).
\]
Different generating sets yield comparable—though not identical—values, so we fix a preferred~$S$ whenever we speak of \emph{the} growth rate.
For a comprehensive treatment of growth rates, see \cite{mann2011groups}.

When $\Gamma$ is a free group $F_r$ of rank~$r\ge2$ and $S$ is the usual free basis,
the Cayley graph is the degree $(2r-1)$ tree $T_{2r-1}$. 
A direct computation shows $\gr(F_r,S)=2r-1$.
If $H\le F_r$ is an arbitrary subgroup we measure its growth by
\[
\gr_{F_r}(H)\;=\; \lim_{n\to\infty} f_{_{(H,F_r,S)}}(n)^{1/n}  \ \in \ [1,2r-1],
\]
%
where $f_{_{(H,F_r,S)}}(n)$ counts the number of words in~$H$ of length~$n$ (in  $T_{2r-1}$ ).
At the extremes, cyclic subgroups have growth~$1$
while finite‐index subgroups share the growth of~$F_r$.
A natural, question is:

\begin{quote}
\emph{Which  values are attained by $\gr_{F_r}(H)$ as $H$ ranges over all subgroups?}
\end{quote}

It was proven in \cite{louvaris2024density} that the set of growth rates of finitely generated subgroups of $F_r$ is dense in the interval $[1,2r-1]$. 
It is natural to ask if every $\alpha \in [1,2r-1]$ arises as the growth rate of some subgroup.
If we restrict to finitely generated subgroups the answer is no, as the growth rate of a finitely generated subgroup is an algebraic number (see Remark ~\ref{remark:growth-is-algebraic}). 

However, if $H \leq F_r$ is allowed to be infinitely generated, we prove:
\begin{theorem}[Main Theorem]\label{thm:main}
    Let $r\ge2$ and $\alpha\in[1,2r-1]$.
    There exists a subgroup $H\le F_r$ whose
    growth rate within~$F_r$ equals~$\alpha$. That is
    \[
    \gr_{F_r}(H)=\alpha.
    \]
\end{theorem}

This is a consequence of the following attractive result, which is of independent interest:
\begin{theorem}\label{thm:increasing_union}
Let $J\rightarrow B$ be an immersion of graphs where $B$ is a finite bouquet of circles.
Let $J_i$ be an increasing sequence of  finite connected based subgraphs with $J=\cup_{i=1}^\infty J_i$.
Let $H=\pi_1J$ and for each $i$, let $H_i=\pi_1J_i$, and regard these as subgroups of $F=\pi_1B$.
Then $$\lim_{i\to \infty} \omega_{_F} (H_i) \ = \ \omega_{_F} (H).$$
\end{theorem}

\subsection{Sketch of Theorem~\ref{thm:main}}
Theorem~\ref{thm:main} is proven in Section~\ref{sec:constructing_graph}, which describes how to create a sequence of subgraphs to which Theorem~\ref{thm:increasing_union} can be applied.

Every subgroup $H\le F_r$ corresponds to a  based  graph~$G$ immersed in the bouquet~$B_r$, and graphs with maximal degree bounded by $2r$ correspond to subgroups of $F_r$. 
When $G$ is finite, the growth rate of~$H$ coincides with the spectral radius of the non–backtracking matrix~$B_G$.
This is explained in Section~\ref{nbm-facts}.

Using the density theorem of~\cite{louvaris2024density} we choose a sequence of \emph{finite} graphs $H_1,H_2,\dots$ whose spectra converge to~$\alpha$.
We  glue these graphs together along an
infinite path, as shown in Figure ~\ref{fig:glueing_graphs}.
It is very intuitive  that the growth rate of the resulting (infinite) graph should be $\alpha$ when the arcs joining them are sufficiently long.

\begin{figure}
    \centering
    \resizebox{1\textwidth}{!}{%
    \begin{circuitikz}
        \tikzstyle{every node}=[font=\Huge]
        \node at (5,15) [circ] {};
        \node at (6.25,15) [circ] {};
        \node at (8.75,15) [circ] {};
        \node at (13.75,15) [circ] {};
        \draw [ dashed] (5,14.07) circle (0.4cm);
        \draw (5,15) to[short] (6.25,15);
        \draw (6.25,15) to[short] (8.75,15);
        \draw (8.75,15) to[short] (13.75,15);
        \draw (13.75,15) to[short] (15.75,15);
        \draw [ dashed] (6.25,13.9) circle (0.6cm);
        \draw [ dashed] (8.75,13.7) circle (0.8cm);
        \draw [ dashed] (13.75,13.5) circle (1cm);
        \node [font=\normalsize] at (8.75,14) {$H_2$};
        \node [font=\normalsize] at (6.25,14) {$H_1$};
        \node [font=\normalsize] at (5,14) {$H_0$};
        \node [font=\normalsize] at (13.75,14) {$H_3$};
        \node [font=\Huge] at (16.25,15) {$...$};
        \draw (5,15) to[short] (5,14.5);
        \draw (6.25,15) to[short] (6.25,14.5);
        \draw (8.75,15) to[short] (8.75,14.5);
        \draw (13.75,15) to[short] (13.75,14.5);
        \end{circuitikz}
    }%
    \caption{Construction of graph with prescribed growth.}
    \label{fig:glueing_graphs}
\end{figure}
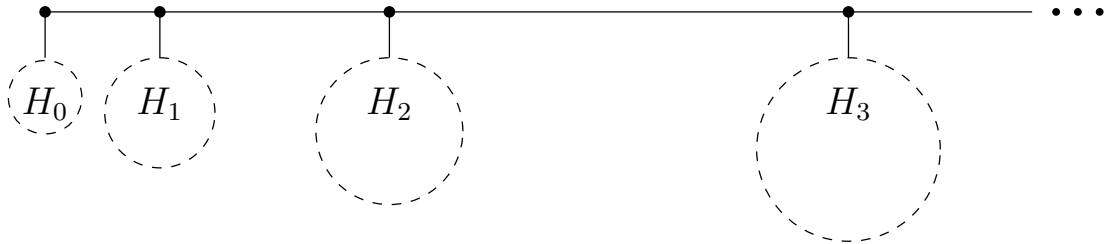

\section{Subgroups of the free group}
Our main tool for studying the subgroups of the free group is through graphs and graph immersions, as  popularized by Stallings \cite{Stallings83}. 

For a locally finite graph $G = (V, E)$ and a vertex $v \in V(G)$, denote by $N_G(v)$ the set of edges with (at least) one endpoint equal to $v$. 

\begin{definition}
    Let $A$ and $B$ be two directed graphs. 
    An \emph{immersion} of graphs is a function $f: A\rightarrow B$ mapping vertices to vertices and edges to edges, which is locally injective: $f_v : N_A(v) \rightarrow N_B(f(v))$ is injective for each $v \in V(A)$. 
\end{definition}

\begin{lemma}\label{lem:immersion pi_1 injective} Let $(A,a)$ and $(B,b)$ be directed based graphs.
Let $f: A\rightarrow B$ be a basepoint preserving immersion. 
Then $f_*: \pi_1(A,a)\rightarrow \pi_1(B,b)$ is injective.
\end{lemma}
\begin{proof}
This holds since the induced cover between universal covers is an immersion of trees, and hence an injection. See  \cite[prop~5.3]{Stallings83}. 
\end{proof}

Let $B_r$ be a graph with one vertex and $r$ directed edges.
Letting the basepoint of $B_r$ be its vertex, we have $\pi_1 B_r\cong F_r$, where $F_r$ is a free group of rank~$r$ whose basis corresponds to the directed loops of $B_r$.
If $G$ is a directed graph with an immersion from $G$ to $B_r$ and $v \in V(G)$ is a vertex, then $\pi_1(G, v) \leq F_r$
by Lemma~\ref{lem:immersion pi_1 injective}.


\begin{theorem} \label{thrm:labeling-infinite-graph}
    Let $G = (V, E)$ be an undirected graph whose degrees are bounded by $2r$.
    Then, $G$ can be labeled and directed so that there is an immersion from $G$ to $B_r$ if and only if every finite subgraph of $G$ can be labeled and directed in such a way. 
\end{theorem}

\begin{remark}
    Theorem \ref{thrm:labeling-infinite-graph} is more general than what needed for our purposes: we construct an infinite graph by gluing together immersed graphs along a long ray, hence the graph inherits that immersion. 
\end{remark}

If $G$ is a finite, undirected graph whose vertex degrees are  $\leq 2r$, then there is an immersion from $G$ to $B_r$. 
See  \cite[p.~57]{Hatcher:478079}, which we revisit in \cite[lem~2.2]{louvaris2024density}. 
Thus we have the following:
\begin{corollary}\label{cor:creating immersions}
    Let $G = (V, E)$ be an undirected graph whose degrees are bounded by $2r$. 
    Then the edges of $G$ can be labeled and directed so that there is an immersion from $G$ to $B_r$. 
\end{corollary}

\begin{proof}[Proof of Theorem \ref{thrm:labeling-infinite-graph}]

The proof is similar to the proof of the de~Bruijn–Erdős Theorem \cite{bruijn1951colour}.
Let $G = (V,E)$ be an undirected graph with maximal degree $2r$.  
Labeling (and ordering) the edges of $G$ means choosing a function $L : E \rightarrow \{1, \ldots, r\} \times \{\pm 1\}$, assigning to each edge of $G$ a letter and a direction. 
We use the notation $L(G)$ for the directed, labeled graph obtained by applying $L$ to $G$. 
A function $L$ is a \emph{proper labeling} of $G$ if there is an immersion $L(G) \rightarrow B_r$. 

Give $\{1, \ldots, r\} \times \{\pm 1\}$ the discrete topology, and define $X$ to be the space of all possible labelings (proper or not) of $G$. 
Then $X$ has the product topology, $X = (\{1, \ldots, r\} \times \{\pm 1\})^{E(G)}$.
By Tychonoff's theorem, $X$ is compact. 

For a finite subgraph $F \subset G$, the following set is closed: 
$$X_F \ = \ \{L \in X : L  \text{ is a proper labeling of } F\}$$
 To see this, observe that a labeling $L$ is proper if and only if every vertex $u \in L(G)$ has neither two outgoing edges nor two incoming edges labeled with the same letter. Hence the following set of non-proper labelings is open: 
\[ X_F^{C} \ = \ \cup_{v \in V(F)} \cup_{e_1 \neq e_2 \in N(v) \times N(v)}  \{L : L(e_1) = L(e_2)\}.\]

Moreover, for finite subgraphs $F_1, \ldots, F_k$ in $G$, we have  
$$X_{F_1} \cap \cdots \cap X_{F_k}
\ = \ 
X_{F_1 \cup \cdots \cup F_m}$$
and since $X_{F_1 \cup \cdots \cup F_m} \neq \emptyset$ by our assumption, the sets $\{X_F : F \text{ is a finite subgraph}\}$ satisfy the finite intersection property. 
The compactness of $X$ implies that $\cap_{F}X_F \neq \emptyset$, which completes the proof, as any  $L \in \cap_{F}X_F$ is a proper labeling for $G$. 
\end{proof}

\subsection{Growth rates of graphs}
Let $G$ be an undirected, connected graph with vertex degrees bounded above by $2r$. Let $L$ be a proper labeling of $G$, and let $v \in G$ be a vertex.
By the discussion above, $\Gamma := \pi_1(L(G),v) \leq \pi_1 (B_r)$. 
We are interested in the growth rate of $\Gamma$ in $\pi_1B_r$, with $\pi_1B_r$ generated by the $r$ labeled edges of $B_r$. 

The map $L$ induces an embedding of every closed non-backtracking loop based at $v$ in $G$ in the Cayley graph of $\pi_1B_r$ (freely generated by the $r$ edges), hence, the growth rate of $\Gamma$ in $F_r$ equals 
\[\gr(G) := \lim_{n \rightarrow \infty} \text{NB}_v(n)^{1/n}, \]
where $\text{NB}_v(n)$ is the number of
 length~$n$ non-backtracking loops based at $v$.
The growth rate does not depend on the vertex $v$.

We approach the proof of Theorem~\ref{thm:main} by constructing a graph $G$ with the appropriate growth rate, as described Section~\ref{sec:constructing_graph}.

\begin{remark}\label{remark:growth-is-algebraic}
    In general, the graph $G$ and the subgroup 
    $\Gamma = \pi_1G$ will not be finitely generated.
    Indeed, there are only countably many finite graphs, so only countably many $\alpha$ could arise. Moreover,
 every finitely generated subgroup $\Gamma \leq F_r$ is $\pi_1$ of a  finite immersed graph $G$, so its growth rate is the Perron-Frobenius eigenvalue of the non-backtracking matrix of $G$, and is thus an algebraic number. 
\end{remark}

\subsection{Linear algebra cheat sheet}\label{nbm-facts}
A major part of our analysis of growth rates of graphs uses Perron-Frobenious Theory for non-negative matrices. 
A detailed treatment of this subject can be found in many texts, e.g.\ \cite{seneta2006non}. 
We mention the relevant facts for our purposes.

Let $T$ be a finite matrix with non-negative entries. 
\begin{definition}\label{def:irreducible_matrix}
    $T$ is \emph{irreducible} if for all $i,j$, there exists $k \in \N$ with $(T^k)_{i,j} > 0$. 
\end{definition}

By the Perron-Frobenious Theorem, the leading eigenvalue $\lambda_1(T)$ of an irreducible non-negative matrix $T$ is a positive  number, with an eigenvector with positive entries.

\begin{lemma}\label{lemma:growth_rate_eigenvalues}
    Let $T$ be an irreducible non-negative matrix.
    Then, 
    \begin{enumerate}
        \item  $\limsup_{k \rightarrow \infty} \big((T^k)_{i,j}\big)^{1/k} = L$ for all $i,j$.
        (So the limit is independent of $i,j$). 
        \label{lemma:growth:first-part}
        \item $L = \lambda_1(T)$, the Perron-Frobenius eigenvalue of $T$. 
    \end{enumerate}
\end{lemma}
\begin{remark}
    If $T$ is aperiodic, we can replace the $\limsup$ with $\lim$.
\end{remark}

Let $G = (V, E)$ be a graph.
Denote by $\overrightarrow{E}=\{(u,v):\{u,v\}\in E\}$ the set of oriented edges of $G$.

\begin{definition}\label{defn_of_non-backtracking}
        The \emph{non-backtracking matrix} of $G$ is the $|\overrightarrow{E}|\times|\overrightarrow{E}|$ matrix  $B_G$  with entries:
        \begin{align*}
            \big(B_G\big)_{e,f} 
            \ = 
            \begin{cases}
                1 & \text{ if } e = (u,v) \text{ and } f = (z,w) \text{ where  } v = z \text{ and } u \neq w \\
                0 & \text{ otherwise.}
            \end{cases}
        \end{align*}       
\end{definition}

If $G$ is a finite connected graph with minimal degree $\geq 2$, that is not a cycle, then $B_G$ is irreducible \cite{glover2020spectral}.
Thus, $\lambda_1(B_G) = \limsup_n (B^n_{e,f})^{1/n} = \gr(G)$
by Lemma ~\ref{lemma:growth_rate_eigenvalues}.

Our goal is to construct a graph with given growth rate. 
As discussed above, the graph will generally  be infinite so we need to extend the theory to infinite graphs and matrices. 
If $G$ is locally finite and $B_G$ is irreducible, the equality $\limsup_n (B^n_{e,f})^{1/n} = \gr(G)$ remains true: $\gr(G) = \lim_{n\rightarrow\infty} \text{NB}_v(n)^{1/n}$ for $v \in V(G)$, and
\begin{align}\label{sum:nbm}
    \text{NB}_v(n) = \sum_{e=(v, \cdot )} \sum_{f = (\cdot, v)} B^n_{e,f} 
\end{align}
so $\limsup_n \text{NB}_v(n)^{1/n}=\max_{e,f}\{\limsup_n (B_{e,f}^{n})^{1/n}\}$. 
Since $G$ is locally finite, the sum in \eqref{sum:nbm} is finite, and since $B$ is irreducible, $\limsup_n (B_{e,f}^{n})^{1/n}$ does not depend on $e,f$, which concludes the result.


Let $T$ be an infinite (countable) matrix, so that $T^k_{i,j} < \infty$ for all $i,j,k$. 
Definition ~\ref{def:irreducible_matrix} extends naturally to the infinite case, and  Lemma~\ref{lemma:growth_rate_eigenvalues}.\eqref{lemma:growth:first-part}  remains true \cite{seneta1967finite}. 
It follows that if $G$ is an infinite, locally finite, connected  graph with minimal degree $\geq 2$, and contains a cycle, then $B_G$ satisfies the condition of Lemma ~\ref{lemma:growth_rate_eigenvalues} and $\gr(G) = \limsup_k \{(B^k_{e,e'})^{1/k}\}$.

The following result will enable us to compute $\gr(G)$ using finite approximations. 
For an infinite matrix $T$ and an integer $n$, denote by $T_n$ the $\{1, \ldots, n\} \times \{1, \ldots, n\}$ submatrix of $T$, and by $\gr(T) = \sup_{k} \{(T^k_{i,j})^{1/k}\}$.


\begin{lemma}[\hbox{\cite[thm~3]{seneta1968finite}}] \label{lemma:finite-approximation}
    Let $T$ be an infinite, irreducible non-negative matrix, with $T^k_{i,j} < \infty$ for all $i, j, k$.  
    Let $\{R_n\}_{n \in \N}$ be a subsequence of $\N$ with $T_{R_n}$ irreducible for all $R_n$. 
    Then $\gr(T) = \sup_{n} \gr(T_{R_n})$.
\end{lemma}

\begin{definition}[Core]
Let $D$ be a  based connected graph.
The \emph{core} of $D$ is  
the smallest connected based subgraph $D'\subset D$ containing all cycles of $D$.
Note that $D'$ is either a single vertex,
or $D'$ contains a cycle and has the property that each vertex has degree~$\geq 2$.
Furthermore, the inclusion $D'\rightarrow D$ induces an isomorphism $\pi_1D'\rightarrow \pi_1D$. See \cite{Stallings83}.
\end{definition}

We now have all the ingredients for proving Theorem~\ref{thm:increasing_union}. 
\begin{proof}[Proof of Theorem \ref{thm:increasing_union}]
Let $J'$ be the based core of $J$.
    If $J'$ is finite, the claim holds trivially, so we assume $J'$ is infinite and locally finite.
    The minimal degree in $J'$ is $2$, and $J'$ contains a cycle.
    The following suffices to prove $B_{J'}$ is irreducible: Any directed edges $i,j$ of $J'$,  are contained in a finite subgraph $X$ of $J'$ with minimal degree $2$ which contains a cycle, but not a cycle itself.
    Thus by \cite[prop~2.3]{glover2020spectral} the non-backtracking matrix of $X$ is irreducible, hence  $B_X^k(i,j) \neq 0$ for some $k$. 
    This implies $B^k_{J'}(i,j) \neq 0$.

   Let $\{J_i\}$ be an increasing sequence subgraphs with $J = \cup_i J_i$. 
    It suffices to prove the claim for their cores $\{J'_i\}$ as they have the same $\pi_1$. 
    For sufficiently large $i$, each $J'_i$ has a cycle and minimal degree $2$, so $\{B_{J'_i}\}$ is a sequence of irreducible truncations of $B_{J'}$, and so we can apply Lemma~\ref{lemma:finite-approximation}.
    Observe that $\omega_F(\pi_1 J'_i) \leq \omega_F(\pi_1 J'_{i+1})$
    since $J'_i$ is a subgraph of $J'_{i+1}$, so we can replace $\sup$~by~$\lim$ in Lemma~\ref{lemma:finite-approximation}. 
\end{proof}

\section{Constructing a graph with a given growth rate}\label{sec:constructing_graph}
The goal of this section is to prove:
\begin{theorem} \label{thrm:graph_with_given_growth}
    Let $r \in \N$ and $\alpha \in (1, 2r-1)$. 
    There exists a graph $G$ with degrees bounded by $2r$, such that $\gr(G) = \alpha$.
\end{theorem}
The statement of Theorem~\ref{thrm:graph_with_given_growth} with $\alpha \in \{1, 2r-1\}$
is also valid as we can let $G$ be an immersed circle or $2r$-regular graph.

The construction is conceptually simple: it was shown in \cite{louvaris2024density} that given $\alpha$ and $\epsilon > 0$, there exists a finite graph $H$ with $|\gr(H) - \alpha| \leq \epsilon$. 

Starting with a sequence of finite graphs $\{H_n\}$ whose growth rates converge to $\alpha$, one can ``glue'' them together along an infinite path, as in Figure ~\ref{fig:glueing_graphs}. 

The first obstacle to implementing the idea is that some of the graphs we obtain by the method in \cite{louvaris2024density} might have a growth rate larger than $\alpha$. 
The growth rate of a graph is at least the growth rate of any of its subgraphs, so we ensure  $\gr(H_n) < \alpha$. 

\begin{lemma}\label{lemma:sequence_of_finite_graphs}
    Let $\alpha \in (1, 2r-1)$. 
    There is a sequence of graphs $\{H_n\}_{n \in \N}$ such that:
    \begin{enumerate}
        \item For all $n$, $H_n$ is connected, the degrees of $H_n$ are in $\{2, \ldots, 2r\}$, and there is at least one vertex $v_n \in V(H_n)$ such that $\deg(v_n) <2r$. 
        \item $\lim_{n \rightarrow \infty} \gr(H_n) = \alpha$ and , $\gr(H_n) < \gr(H_{n+1}) < \alpha$ for all $n$. 
    \end{enumerate}
\end{lemma}


\begin{proof}
    Let $\beta \in (1,2r-1)$. 
    By
    \cite[thm~3.2]{louvaris2024density}, there exists a sequence of graphs $\{T_k\}_{k \in \N}$ such that $|\gr(T_k) - \beta| \leq \frac{C}{\log k}$ for some constant $C$.     
    We construct the sequence $H_n$.
    To get $H_1$, invoke the Theorem with $\beta_1 = \alpha  - \epsilon_1$ for some $\epsilon_1 > 0$. 
    We get a sequence $\{T_k\}$ so that $|\gr(T_k) - \alpha + \epsilon_1| \leq \frac{C}{\log k}$, thus there is some $k$ with $T_k$ satisfying $\gr(T_k) < \alpha$.
    Set $H_1$ to be this $T_k$. 
    Observe that \cite[thm~3.2]{louvaris2024density} guarantees that the degrees of $T_k$ are in $\{2, \ldots, 2r\}$ and that $T_k$ is connected. 
    Since $\gr(T_k) < \alpha < 2r-1$, there is a vertex with degree $<2r$. 
    
    Continue the construction similarly. 
    To get $H_n$, invoke the Theorem with $\beta_n = \alpha - \epsilon_n$, where $0 < \epsilon_n < \alpha - \gr(H_{n-1})$ and $\epsilon_n \rightarrow 0$. 
\end{proof}

Another ingredient in the gluing-along-infinite-path construction is gluing. 
But gluing where? Lemma ~\ref{lemma:new_graph} gives a quantitative answer. 

\begin{lemma}\label{lemma:new_graph}
    Let $G_1$ and $G_2$ be two finite graphs, whose minimal degree is at least $2$ and where both $G_1$ and $G_2$ are not cycles. 
    Create a new graph $G_3$ by connecting a path of length $k$ to $G_1$ from one end, and to $G_2$ from the other end. 
    Then $\gr(G_3) = \max\{\gr(G_1), \gr(G_2)\} + \epsilon(k)$, where $\epsilon(k) \rightarrow 0$ as $k \rightarrow \infty$.
    More precisely, $\epsilon(k) \leq C/m^k$, where $m = \max\{\gr(G_1), \gr(G_2)\}$. 
\end{lemma}

\begin{proof}
The proof uses linear algebra.
By the discussion in~\ref{nbm-facts}, $\gr(G) = \lambda_1(B_G)$. 

We use the Collatz-Wielandt formula for non-negative matrices (see e.g.\  \cite[p.~31]{berman1994nonnegative}):
Let $A$ be an irreducible, $n \times n$ real matrix such that $A_{i,j} \geq 0$, and let $x \in \R^n$ be a vector such that $x_i > 0$. 
Then $\lambda_1(A) \leq \max_{i \in [n]} \frac{(Ax)_i}{x_i}$. 

Let $u \in G_1$ and $v \in G_2$ be the vertices in $G_1$ and $G_2$ which are connected to the $k$-path $u = w_0, w_1, \ldots, w_k = v$, let $B_1, B_2,B_3$ be the non-backtracking matrices of $G_1, G_2, G_3$ respectively, and let $x_1$ and $x_2$ be the Perron–Frobenius eigenvectors of $B_1$ and $B_2$; i.e. $x_{1,i},x_{2,i} > 0$ and $B_1x_1 = \lambda_1(B_1)x_1, B_2x_2 = \lambda_1(B_2)x_2$. 
By \cite[prop~2.3]{glover2020spectral}, both $B_1$ and $B_2$ are irreducible.

Denote the maximal eigenvalue $m = \max \{\lambda_1(B_1), \lambda_1(B_2)\}$. Define the vector $\Tilde{x}$, indexed by directed edges in $G_3$ as follows
where we choose $C_0, C_1$ later: 
\begin{align*}
    \Tilde{x}(e) = \begin{cases}
        x_1(e) & \text{ if } e \in G_1 \\
        x_2(e) & \text{ if } e \in G_2  \\ 
        m^{-t-k-1} C_0 & \text{ if } e = (w_{t-1}, w_t) \\
        m^{-t}C_1 & \text{ if } e = (w_t, w_{t-1})
    \end{cases}
\end{align*} 
Observe that this gives a non-zero value for every directed edge of $G_3$. 
Our goal now is to bound $ \max_{e \in G_3} f(e)$ with $f(e) := \frac{B_3\Tilde{x}(e)}{\Tilde{x}(e)}$. 

Let $e \in G_1$ such that both endpoints of $e$ are different from $u$. 
Then by our definition $f(e) = \lambda_1(B_1)$. 
Similarly, if $e \in G_2$ such that both endpoints are different from $v$, then $f(e) = \lambda_1(B_2)$. 

Let $e = (x,u)$ be an edge entering $u$ (not from the added path).
Then 
$$ \frac{B_3\Tilde{x} (e)}{\Tilde{x}(e)}  
\ = \ 
\frac{\sum_{e^{-1} \neq h = (u,t)} \Tilde{x}(h) + \Tilde{x}((u,w_1))}{\Tilde{x}(e)} 
\ =  \
\lambda_1(B_1) + \frac{m^{-k}C_0}{\Tilde{x}(e)} .$$

Let $e = (w_{t-1}, w_t)$ be an edge on the path with $w_{t+1} \neq v$. 
Then $$\frac{B_3 \Tilde{x}(e)}{\Tilde{x}(e)}
\ = \ 
\frac{m^{-t-k} C_0}{m^{-t-k-1} C_0} \ = \  m.$$ 

Let $e = (w_{k-1},v)$. 
\begin{align} \label{quantitiy}
    \frac{B_3 \Tilde{x}(e)}{x(e)} 
\ = \ 
\frac{\sum_{(v,t) = h \neq e^{-1}} \Tilde{x}(h)}{C_0 m}  \text{ ,}
\end{align}
and pick $C_0$ so that \eqref{quantitiy} is smaller than $m$ (this does not depend on $k$). 

Let $e = (t,v)$ be an edge entering $v$. Then 
$$ \frac{B_3\Tilde{x} (e)}{\Tilde{x}(e)} 
\ = \ 
\frac{\sum_{e^{-1} \neq h 
\ = \ 
(t,v)} \Tilde{x}(h) + \Tilde{x}((v,w_{k-1}))}{\Tilde{x}(e)}  
\  =\  
\lambda_1(B_2) + \frac{m^{-k}C_1}{\Tilde{x}(e)} .$$

Let $e = (w_t, w_{t-1})$ be an edge on the path returning to $u$ such that $w_{t-1} \neq u$. Then a similar computation gives  $f(e) = m$. 

Let $e = (w_1,u)$ be the edge entering $u$ from the path. 
Then 
\begin{align}\label{equation:quatity2}
    \frac{B_3\Tilde{x} (e)}{\Tilde{x}(e)} 
    \ = \ \frac{\sum_{e^{-1} \neq h = (u,t)} \Tilde{x}(h) + \Tilde{x}((u,w_1))}{\Tilde{x}(e)}  
    \ =\ 
    \frac{\sum_{e^{-1}\neq h = (u,t)} \Tilde{x}(h)}{C_1m} ,
\end{align}
and we choose $C_1$ so that \eqref{equation:quatity2} is smaller than $m$. 

Overall, $\max_{e} f(e) \leq m + C'm^{-k}$ for some constant $C'$, which proves the claim. 
\end{proof}

\begin{remark}
    The constant $C' = C'(G_1, G_2)$ depends on the graphs $G_1, G_2$.
    We invoke Lemma ~\ref{lemma:new_graph} thinking of $G_1$ and $G_2$ as constants, while the parameter $k$ goes to $\infty$. 
\end{remark}

\subsection{Construction of \texorpdfstring{$G$}{G}}
We construct $G$ in steps. 
Our main ingredient is the sequence of graphs $\{H_n\}_{n \in \N}$ given by Lemma~\ref{lemma:sequence_of_finite_graphs}.
Let $\alpha_n := \gr(H_n)$, with $\lim_{n\rightarrow \infty} \alpha_n = \alpha$. 
Let $G_0$ be a graph with one node, labeled with 
$1$, and no edges.

Let $v_1 \in V(H_1)$ be a vertex whose degree is less than $2r$ in $H_1$. 
Set $G_1$ to be the graph obtained by ``gluing'' $H_1$ to $G_0$ by adding an edge from $v_1$ to the vertex labeled $1$ in $G_0$. 
More precisely, $V(G_1) = V(G_0) \cup V(H_1)$ and $E(G_1) = E(G_0) \cup E(H_1) \cup \{1, v_1\}$. 
The growth rate of $G_1$ is equal to the growth rate of $H_1$. 

We proceed to construct $G_2$.
By Lemma~\ref{lemma:sequence_of_finite_graphs}, $ \alpha_1 < \alpha_2 < \alpha_3$.
We attach edges to the vertex $1 \in G_1$: $\{1,2\}, \{2,3\} ,\ldots \{k_1-1, k_1\}$ (where all the vertices labeled with integers are new vertices), and attach the vertex $k_1$ to a vertex $v_2 \in H_2$ with degree less than $2r$. 
By Lemma \ref{lemma:new_graph}, we can choose $k_1$ so that $\gr(G_2) \in (\alpha_2, \alpha_3)$. 

We construct $G_n$ similarly.
The important properties of $\{G_n\}$ are that:
\begin{enumerate}
\item $G_n$ is connected and the degrees of $G_n$ are in $\{2, \ldots, 2r\}$.
\item $\gr(G_n) \in (\alpha_{n}, \alpha_{n+1})$.
\item $G_n$ is a subgraph of $G_{n+1}$.
\end{enumerate}

The graph $G$ will be the outcome of this process. 
More formally, $V(G) = \cup_{n} V(G_n)$ and $E(G) = \cup_{n} E(G_n)$.

\begin{theorem}\label{thrm:growth-of-big-graph}
    $\gr(G) = \alpha$. 
\end{theorem}
\begin{proof}
This holds by Theorem~\ref{thm:increasing_union}.
\end{proof}



\section{Growth Rates and Spectra of Infinite Graphs?}
Let $B$ be an infinite, non-negative, irreducible matrix as in Subsection~\ref{nbm-facts}, corresponding (in some sense) to a strongly connected directed graph $G$ --- for instance, $B$ may be the adjacency matrix of $G$.  
Suppose that $B$ has uniformly bounded row and column degrees, so  $B$ defines a bounded linear operator on $\ell^p$ for some $p \in [1, \infty]$. A celebrated result in functional analysis, the Krein–Rutman Theorem, states that if $B$ is compact then the spectral radius of $B$ is an eigenvalue. But if $B$ is the adjacency matrix of a directed graph, it is compact iff the number of edges of $B$ is finite. 

\begin{question}
    Under what 
    nontrivial assumptions on $G$, does the growth rate $\gr(B)$ correspond to an eigenvalue of $B$?  
    When is $\gr(B)$ equal to the spectral radius of $B$ as an operator on $\ell^p$?
\end{question}
Using Lemma~\ref{lemma:finite-approximation}, one can show that $\gr(B)$ always belongs to the approximate spectrum of $B$.
Moreover, if $B$ is self-adjoint (and hence $G$ is undirected), then $\gr(B)$ is exactly the spectral radius of $B$ acting on $\ell^2$; see \cite{mohar1989survey}.



\bibliographystyle{abbrv} 
\bibliography{bib}

\end{document}